\documentclass[final,leqno,onefignum,onetabnum]{siamart1116}

\usepackage{amssymb}
\usepackage{mathtools}
\usepackage{tikz}
\usepackage{pgfplots}
\usepackage{stmaryrd}
\usepackage{pgfplotstable}
\usepackage[caption=false]{subfig}
\usepackage{array}
\usepackage{cite}
\usepackage{comment}
\usetikzlibrary{calc}

 \DeclarePairedDelimiter{\parens}{\lparen}{\rparen}

\hypersetup{colorlinks=false}

\newcommand{\A}[0]{\mathcal{A}}

\newcommand{\Bf}[0]{ \A}

\newcommand{\Eps}[0]{\boldsymbol{\varepsilon}}
\newcommand{\Kappa}[0]{\boldsymbol{K}}

\newcommand{\Mom}[0]{\boldsymbol{M}}
\newcommand{ \leikkaus}[0]{\boldsymbol{Q}}
\newcommand{\Div}[0]{\mathrm{div}\,}

\newcommand{\Vn}[0]{V_{n}}
\newcommand{\mnn}[0]{M_{nn}}
\newcommand{\dn}[1]{\frac{\partial #1}{\partial \boldsymbol{n}}}
\newcommand{\dnt}[1]{\tfrac{\partial #1}{\partial \boldsymbol{n}}}
\newcommand{\qn}[0]{Q_n}
\newcommand{\effs}[0]{V_n}

\newcommand{\mns}[0]{M_{ns}}
\newcommand{\msn}[0]{M_{sn}}

\newcommand{\Bh}[0]{\mathcal{A}_h}
\newcommand{\Lh}[0]{\mathcal{L}_h}

\newcommand{\osc}[0]{\mathrm{osc}}

\newcommand{\Ch}[0]{\mathcal{C}_h}
\newcommand{\Eh}[0]{\mathcal{E}_h}
\newcommand{\Gh}[0]{\mathcal{G}_h}

\newcommand{\vertiii}[1]{{\left\vert\kern-0.25ex\left\vert\kern-0.25ex\left\vert #1 
    \right\vert\kern-0.25ex\right\vert\kern-0.25ex\right\vert}}

\newcommand{\enorm}[1]{\vertiii{#1}}

\newcommand\gvh{g_{E,h} ^v}
\newcommand\grh{g_{E,h} ^r}
\newcommand\fh{f_h}

\newcommand{\tj}[0]{\varepsilon^v_i}  
\newcommand{\tjE}[0]{\varepsilon^v_E}
\newcommand{\kj}[0]{\varepsilon^r_i}  
\newcommand{\kjE}[0]{\varepsilon^r_E}
\newcommand{\gvi}[0]{g^v_i}
\newcommand{\gvE}[0]{g_E^v}
\newcommand{\gmi}[0]{g^r_i}
\newcommand{\gmE}[0]{g_E^r}
\newcommand{\gci}[0]{g^c_i}
\newcommand{\eci}[0]{\varepsilon ^c_i}

\newcommand\Rc{R_i^c}
\newcommand\Rv{R_E^v}
\newcommand\Rr{R_E^r}
\newcommand\R{\mathcal{R}}

 \newcommand\B{ \A}

\newtheorem{thm}{Theorem}

\newtheorem{rem}{Remark}

\theoremstyle{definition}

\newtheorem{ass}{Assumption}

\numberwithin{equation}{section}

\title{Nitsche's method for Kirchhoff plates\thanks{Submitted to the editors on \today.\funding{This work was supported by the Academy of Finland (Decision 324611) and by the Portuguese government through FCT (Funda\c{c}\~ao para a Ci\^encia e a Tecnologia), I.P., under the projects PTDC/MAT-PUR/28686/2017 and UTAP- EXPL/MAT/0017/2017.}}}
\author{Tom Gustafsson\thanks{Department of Mathematics and Systems Analysis,
Aalto University, P.O. Box 11100,  \hfill \break00076 Aalto, Finland
e-mail: 
(\email{tom.gustafsson@aalto.fi}, \email{rolf.stenberg@aalto.fi}). }  \and Rolf Stenberg\footnotemark[2] 
 \and {Juha Videman}\thanks{CAMGSD/Departamento de Matem\'atica, Instituto Superior T\'ecnico, Universidade   de Lisboa, 1049-001 Lisbon, Portugal (\email{jvideman@math.tecnico.ulisboa.pt}). }  }
\begin{document}

\maketitle

\begin{abstract}
  We introduce a Nitsche's method for the numerical approximation of the Kirchhoff--Love plate equation under general Robin-type boundary conditions.  We
  analyze the method by presenting a priori and a posteriori error estimates
  in mesh-dependent norms.  Several numerical examples are given to
  validate the approach and demonstrate its properties.
\end{abstract}

\begin{keywords} 
    Kirchhoff plate, Nitsche's method
\end{keywords}

\begin{AMS}65N30\end{AMS}

    \section{Introduction}

  Implementation of $H^2$-conforming finite element methods can be a challenge
  due to the $C^1$-continuity requirement of the finite element
  basis~\cite{ciarlet-2002-finit-elemen}.  In fact, it is a common motivation
  for developing discontinuous Galerkin techniques~\cite{brenner-2012-isopar-c}
  where it is sufficient to guarantee the conformity in a weak sense only,
  other nonconforming methods using special finite elements~\cite{BdVNS-I,BdVNS-II}, or
  mixed methods~\cite{arnold19_hellan_herrm_johns} where the fourth-order
  problem is split into a system of lower order problems. At the same
  time, however, finite element codes including classical $H^2$-conforming elements---such
  as, e.g., the Argyris triangle and the rectangular Bogner--Fox--Schmit element~\cite{ZT}---abound and many are free and readily
  available~\cite{dominguez-2008-algor, renard-2020-getfem,
    gustafsson-2020-kinnal-fem, valdmanc1} to be used in the discretization of fourth-order
  differential operators. Thus, the main challenge
  remaining for the end user is the proper implementation of external loads and
  boundary conditions.

  In \cite{nitsche1971variationsprinzip}, Nitsche introduced a consistent penalty-type method for imposing Dirichlet boundary conditions in the second-order Poisson problem. Nitsche's method was extended to
 other boundary conditions (in particular, inhomogeneous Robin) in Juntunen--Stenberg~\cite{juntunen2009nitsche} by unifying the implementation and analysis
  via
  a parameter-dependent boundary value problem;
  an improved a priori analysis was presented in Lüthen--Juntunen--Stenberg~\cite{LJS}.
  Different boundary conditions (Dirichlet, Neumann, Robin) were obtained
  by changing the value of a single nonnegative parameter.
  The resulting method performed
  similarly well in all cases, i.e.~altering the
  parameter value did not deteoriate the conditioning of the
  resulting linear system or lead to an overrefinement
  as in traditional methods.

   In this study we explore the above ideas \cite{nitsche1971variationsprinzip, juntunen2009nitsche, LJS} in the context of fourth-order $H^2$-conforming problems.  In particular, we seek to unify the implementation and the analysis of
  different boundary conditions for the Kirchhoff--Love plate equation \cite{Kirchhoff1850, love1888xvi} by
  presenting a Nitsche's method which incorporates the boundary conditions in
  the discrete formulation as consistent penalty terms.  We consider elastic
  Robin-type boundary conditions for the deflection and the rotation
  including applied external forces and moments. The classical boundary conditions for the Kirchhoff plates  (clamped,
  simply supported and free) are recovered as special cases.
  Moreover, we allow general matching conditions at the corners of the domain so
  that ball supports, point forces and springs \cite{Friedrichs, blaauwendraad10_plates_fem, zhang19_accur} are all covered by the same
  formalism.

  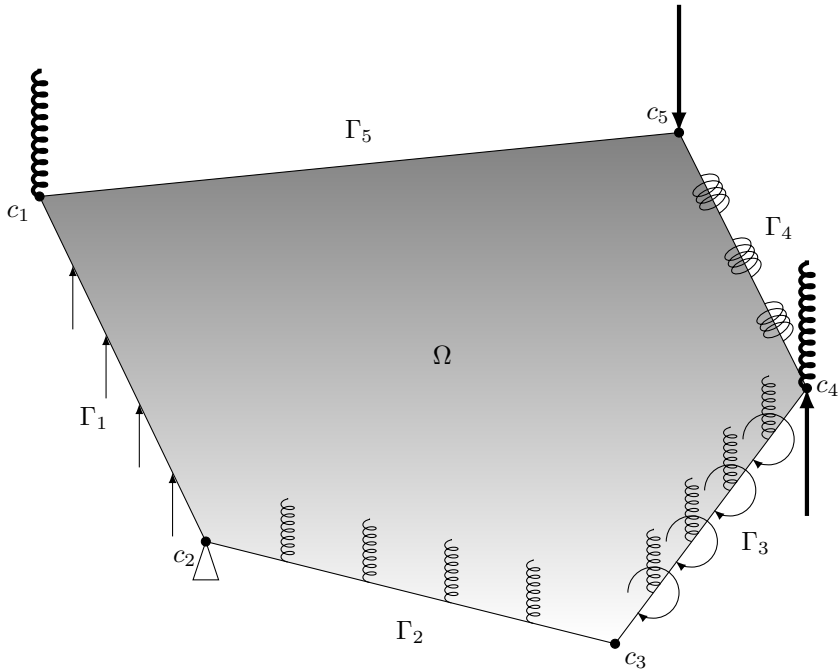
\begin{figure}[h]
    \centering
    \begin{tikzpicture}[scale=1.7]
        \coordinate (A) at (-0.5,0);
        \coordinate (B) at (0.8,-2.7);
        \coordinate (C) at (4,-3.5);
        \coordinate (D) at (5.5,-1.5);
        \coordinate (E) at (4.5,0.5);
        \filldraw[black!10!white,draw=black,shade] (A) -- (B) -- (C) -- (D) -- (E) -- (A);
        \draw [fill=black] (A)node[below left] {$c_1$} circle (1pt);
        \begin{scope}[shift={(A)}]
            \draw[decoration={segment length=1.5mm,coil},decorate,ultra thick] (0,0) -- (0,1);
        \end{scope}
        \draw [fill=black] (B)node[below left] {$c_2$} circle (1pt);
        \begin{scope}[shift={(B)}]
            \draw (0,0) -- (-0.1,-0.3) -- (0.1,-0.3) -- (0,0);
        \end{scope}
        \draw [fill=black] (C)node[below right] {$c_3$} circle (1pt);
        \draw [fill=black] (D)node[right] {$c_4$} circle (1pt);
        \begin{scope}[shift={(D)}]
          \draw[ultra thick,-latex] (0, -1) -- (0, 0);
          \draw[decoration={segment length=1.5mm,coil},decorate,ultra thick] (0,0) -- (0,1);
        \end{scope}
        \draw [fill=black] (E)node[above left] {$c_5$} circle (1pt);
        \begin{scope}[shift={(E)}]
          \draw[ultra thick,-latex] (0, 1) -- (0, 0);
        \end{scope}
        \begin{scope}[shift={($(B)!0.5!(E)$)}]
          \draw [fill=black] (0.0, 0.0)node[below] {$\Omega$};
        \end{scope}
        \begin{scope}[shift={($(A)!0.7!(B)$)}]
          \draw [fill=black] (-0.3, 0)node[above left] {$\Gamma_1$};
        \end{scope}
        \begin{scope}[shift={($(B)!0.5!(C)$)}]
          \draw [fill=black] (0, -0.15)node[below] {$\Gamma_2$};
        \end{scope}
        \begin{scope}[shift={($(C)!0.5!(D)$)}]
          \draw [fill=black] (0.35, -0.05)node[below] {$\Gamma_3$};
        \end{scope}
        \begin{scope}[shift={($(D)!0.6!(E)$)}]
          \draw [fill=black] (0.2, 0.05)node[right] {$\Gamma_4$};
        \end{scope}
        \begin{scope}[shift={($(A)!0.5!(E)$)}]
          \draw [fill=black] (0.0, 0.1)node[above] {$\Gamma_5$};
        \end{scope}
        \begin{scope}[shift={($(A)!0.2!(B)$)}]
          \draw[-latex] (0, -0.5) -- (0, 0);
        \end{scope}
        \begin{scope}[shift={($(A)!0.4!(B)$)}]
          \draw[-latex] (0, -0.5) -- (0, 0);
        \end{scope}
        \begin{scope}[shift={($(A)!0.6!(B)$)}]
          \draw[-latex] (0, -0.5) -- (0, 0);
        \end{scope}
        \begin{scope}[shift={($(A)!0.8!(B)$)}]
          \draw[-latex] (0, -0.5) -- (0, 0);
        \end{scope}
        \begin{scope}[shift={($(B)!0.2!(C)$)}]
          \draw[decoration={segment length=1mm,coil},decorate] (0,0) -- (0,0.5);
        \end{scope}
        \begin{scope}[shift={($(B)!0.4!(C)$)}]
          \draw[decoration={segment length=1mm,coil},decorate] (0,0) -- (0,0.5);
        \end{scope}
        \begin{scope}[shift={($(B)!0.6!(C)$)}]
          \draw[decoration={segment length=1mm,coil},decorate] (0,0) -- (0,0.5);
        \end{scope}
        \begin{scope}[shift={($(B)!0.8!(C)$)}]
          \draw[decoration={segment length=1mm,coil},decorate] (0,0) -- (0,0.5);
        \end{scope}
        \begin{scope}[shift={($(C)!0.2!(D)$)}]
          \draw[-latex] (180:0.2) arc (180:-130:0.2);
          \draw[decoration={segment length=1mm,coil},decorate] (0,0) -- (0,0.5);
        \end{scope}
        \begin{scope}[shift={($(C)!0.4!(D)$)}]
          \draw[-latex] (180:0.2) arc (180:-130:0.2);
          \draw[decoration={segment length=1mm,coil},decorate] (0,0) -- (0,0.5);
        \end{scope}
        \begin{scope}[shift={($(C)!0.6!(D)$)}]
          \draw[-latex] (180:0.2) arc (180:-130:0.2);
          \draw[decoration={segment length=1mm,coil},decorate] (0,0) -- (0,0.5);
        \end{scope}
        \begin{scope}[shift={($(C)!0.8!(D)$)}]
          \draw[-latex] (180:0.2) arc (180:-130:0.2);
          \draw[decoration={segment length=1mm,coil},decorate] (0,0) -- (0,0.5);
        \end{scope}
        \begin{scope}[shift={($(D)!0.33!(E)$)}]
          \draw[decoration={segment length=1mm,amplitude=2mm,coil},decorate] (0,0)--(0.15,-0.3);
        \end{scope}
        \begin{scope}[shift={($(D)!0.58!(E)$)}]
          \draw[decoration={segment length=1mm,amplitude=2mm,coil},decorate] (0,0)--(0.15,-0.3);
        \end{scope}
        \begin{scope}[shift={($(D)!0.83!(E)$)}]
          \draw[decoration={segment length=1mm,amplitude=2mm,coil},decorate] (0,0)--(0.15,-0.3);
        \end{scope}
    \end{tikzpicture}
    \caption{Definition sketch of the plate with different boundary conditions and elastic supports:
    springs at $c_1$ and $c_4$, a ball support at $c_2$, applied point forces
    at $c_4$ and $c_5$,
    an applied shear force on $\Gamma_1$,
    a spring support on $\Gamma_2$ and $\Gamma_3$,
    an applied torque on $\Gamma_3$,
    a torsion spring support on $\Gamma_4$.}
    \label{fig:plate}
  \end{figure}
  
  The Nitsche method is not only practical to implement but has also
  other advantages. For a very stiff support, i.e.~with almost
  clamped conditions, the traditional method leads to two potential problems: 1) the corresponding stiffness matrix becomes ill-conditioned and 2) the standard a posteriori estimators lead to overrefinement. As for the Poisson problem, these  phenomena can be avoided using the Nitsche method presented in this work.
  Moreover, if one is using plate elements in which second derivatives are included as degrees-of-freedom, e.g., the Argyris
  triangle and the Bogner--Fox--Schmit element, and
  if the boundary conditions are enforced by eliminating degrees-of-freedom, one must verify separately that the
  second-order derivatives are zero along the boundary of the domain.
  In practice, e.g.,~in case of non-right angles, this introduces additional linear
  constraints for the solution to satisfy. 
  Nitsche's method circumvents this issue by enforcing the
  boundary conditions weakly.

The rest of the paper is organized as follows.
In Section~\ref{sec:model}, we introduce
the Kirchhoff plate bending model and its
boundary conditions.
In Section~\ref{sec:nitsche}, we
derive the Nitsche method by augmenting
the model's weak formulation with consistent
penalty-type terms.
In Section~\ref{sec:apriori}, we prove the
stability of the resulting discrete formulation
and present the ensuing a priori
error estimate.
In Section~\ref{sec:aposteriori}, we
present the residual a posteriori error
estimators and
prove an error estimate
via a saturation
assumption.
Finally in Section~\ref{sec:results}, we
demonstrate the approach by performing
computational experiments.

\section{The Kirchhoff plate model}

\label{sec:model}

We start by recalling the Kirchhoff plate model with general boundary conditions, cf. \cite{Friedrichs,NH,FS}. Let $\Omega\subset  \mathbb{R}^2$ be a polygonal domain, with corners $c_i$,  and the boundary $\partial \Omega=\cup_{i=1}^m \Gamma_i $, $i=1,\dots , m$, where each $\Gamma_i$ is a line segment, see Figure~\ref{fig:plate}.   Given the deflection $u:\Omega\to \mathbb{R} $ of the midsurface of the plate, the
curvature $\Kappa$ is defined through 
\begin{equation}\label{curvdef}
  \Kappa(u) = -\Eps(\nabla u), \end{equation}
where the infinitesimal strain $\Eps$ is given by
\begin{equation}
    \Eps(\boldsymbol{v})
    = \frac12 \parens*{\nabla \boldsymbol{v}+\nabla \boldsymbol{v}^T}, \quad (\nabla \boldsymbol{v})_{ij} = \frac{\partial v_i}{\partial x_j}, \quad i,j = 1,2.
\end{equation}
The moment tensor $\Mom$ is given by
the  constitutive relation
\begin{equation}
  \qquad \Mom(u) = \frac{Ed^3}{12(1 + \nu)}
      \parens*{\Kappa(u)
               + \frac{\nu}{1 - \nu}(\text{tr}\,\Kappa(u)) \boldsymbol{I}},
\end{equation}
where $d$ denotes the plate thickness and $\boldsymbol{I}$ is the identity tensor.
Above, $E$ and $\nu$ are the Young's modulus and the Poisson ratio, respectively. 

 The shear force 
$ \leikkaus$ is related to the moment tensor through the moment equilibrium equation
\begin{equation}
     \boldsymbol{\Div} \Mom(u) =  \leikkaus(u),
\end{equation}
where $\boldsymbol{\Div}$ is the vector-valued divergence operator.
The transverse shear equilibrium reads as follows
\begin{equation}
 -\Div  \leikkaus(u) = f
\end{equation}
where $f$ is an external transverse loading.
Combining the above expressions yields the Kirchhoff--Love plate
equation
\begin{equation}
  \label{strong}
 D \Delta^2 u = f,
\end{equation}
where $D$, the plate rigidity, is defined as
\begin{equation}
  D=\frac{E d^3}{12(1-\nu^2)} .
\end{equation} 

We consider quite general boundary conditions. A vertical force $\gvi$ and a normal moment $\gmi$ act on each segment $\Gamma_i$ of the boundary and the support is elastic with respect to both the deflection and the rotation, with the spring constants $1/\tj$ and $1/\kj$, respectively.
At the corner $c_i $, also connected to a spring with constant $1/\eci$, acts a point force $\gci$.

The energy of the system can be written as
\begin{equation}
\begin{aligned}
I(v)&=\frac{1}{2}\int_\Omega \Mom(v) : \Kappa(v) \,\mathrm{d}x \\
&\phantom{=}~+\sum_{i=1}^m  \bigg\lbrace \frac{1}{2\tj }\int_{  \Gamma_i}   v^2\, \mathrm{d} s + \frac{1}{2\kj}\int_{ \Gamma_i } \Big( \dn{v} \Big)^2 \mathrm{d} s +\frac{1}{2 \eci} v(c_i)^2 \\
&\qquad \qquad ~~- \int_{ \Gamma_i}    \gvi v \, \mathrm{d} s 
+  \int_{\Gamma_i}  \gmi    \dn{v}  \, \mathrm{d} s
- \sum_{i=1}^m \gci v(c_i)\bigg\rbrace
-
 \int_\Omega fv \,\mathrm{d}x ,
\end{aligned}
\end{equation}
from where follows the 
  variational formulation: find $u\in
H^2(\Omega)$ such that
\begin{equation}\label{heikko}
\begin{aligned}
  &\int_\Omega \Mom(u)  : \Kappa(v) \,\mathrm{d}x 
+\sum_{i=1}^m  \bigg\lbrace \frac{1}{ \tj }\int_{  \Gamma_i}   u v\, \mathrm{d} s   + \frac{1}{ \kj }\int_{ \Gamma_i }  \dn{u}   \dn{v}   \mathrm{d} s +\frac{1}{ \eci } u(c_i) v(c_i)\bigg\rbrace \\
&=  \int_\Omega fv \,\mathrm{d}x + \sum_{i=1}^m  \bigg\lbrace \int_{ \Gamma_i}    \gvi v \, \mathrm{d} s 
-  \int_{\Gamma_i}  \gmi    \dn{v}  \, \mathrm{d} s
+\sum_{i=1}^m \gci v(c_i)\bigg\rbrace
\quad  \forall v \in H^2(\Omega).
\end{aligned}
\end{equation}
The corresponding boundary value problem is posed using the normal shear force, and the normal and twisting moments
\begin{equation}
\begin{aligned}
  \qn(w)&= \leikkaus(w)\cdot \boldsymbol{n},\\
  \mnn(w)&=\boldsymbol{n} \cdot  \Mom(w) \boldsymbol{n}, \\
  \mns(w)&=\msn(w)=\boldsymbol{s} \cdot  \Mom(w) \boldsymbol{n}.
\end{aligned}
\end{equation}
Above~$\boldsymbol{n}$ denotes the outward unit normal on $\partial \Omega$
and $\boldsymbol{s} = (n_1, -n_2)$ is the respective unit tangent vector. Moreover, we define
the Kirchhoff shear force as
\begin{equation}\label{ksf}
\effs(w)= \qn(w) +\frac{ \partial \mns(w)}{\partial s},
\end{equation}
and the jump in the twisting moment
 \begin{equation}
 \llbracket \mns(u) \rrbracket|_{c_i} = \lim_{\epsilon\rightarrow 0+}\Big(\mns(u)|_{c_i + \epsilon(c_{i+1} - c_i)} - \mns(u)|_{c_i + \epsilon(c_{i-1} - c_i)}\Big),  
 \end{equation}
$i=1,\dots,m,$ with $c_{m+1}=c_1$. 
 After repeated integrations by parts, one then obtains
\begin{equation}\label{intbyparts}
\begin{aligned}
 \int_\Omega D\Delta^2 uv\,\mathrm{d}x   = &\int_\Omega \Mom(u) : \Kappa(v) \,\mathrm{d}x  \\
 &+ \sum_{i=1}^m \bigg\lbrace \int_{\Gamma_i} \mnn(u) \dn{v}\,\mathrm{d}s  \\
 &\qquad ~~~~ -  \int_{\Gamma_i} \Vn(u)v\,\mathrm{d}s   -
 \llbracket \mns(u) \rrbracket|_{c_i} v(c_i)\bigg\rbrace.
 \end{aligned}
\end{equation}
 Substituting \eqref{intbyparts} into the weak form \eqref{heikko} leads to the differential equation 
  \eqref{strong}  and the boundary conditions 
on $\Gamma_i$  
\begin{equation}
  \label{eq:boundaryconds}
  \Vn(u) + \frac{1}{ \tj} u = \gvi, \quad \mnn(u) - \frac{1}{ \kj} \dn{u} = \gmi,  \ i=1,\dots,m,
\end{equation}
and the corner conditions  
\begin{equation}
  \label{eq:matchingcond}
  \llbracket \mns(u) \rrbracket |_{c_i}  + \frac{1}{\eci} u(c_i)
  = \gci, \quad   i=1,\ldots, m,
\end{equation}
with the letters $v,r,c$ indicating vertical, rotational and corner, respectively. The boundary value problem is now formed by the equations
\eqref{strong}, \eqref{eq:boundaryconds} and \eqref{eq:matchingcond}.

\begin{rem}
  \label{rem:stdbcs}
The boundary and corner conditions \eqref{eq:boundaryconds} and
\eqref{eq:matchingcond} include
\begin{itemize}
  \item \emph{clamped} at edge $\Gamma_i$ when $ \tj,  \kj, \eci, \varepsilon ^c_{i+1}  \rightarrow 0$,
\item \emph{simply supported} at edge $\Gamma_i$ when $ \tj, \eci ,\varepsilon ^c_{i+1}  \rightarrow 0$,
  $ \kj \rightarrow \infty$, $\gmi = 0$,
\item \emph{free} at edge $\Gamma_i$ when $ \tj,  \kj, \eci,\varepsilon ^c_{i+1}  \rightarrow \infty$ and $\gvi = \gmi = \gci = g_{i+1}^c= 0$,
\end{itemize}
and various other combinations of prescribed forces and moments on the boundary
and at the corners of the domain.
\end{rem}

\section{The finite element method}

\label{sec:nitsche}

The domain $\Omega$ is split into non-overlapping regular elements $K \in \Ch$. As usual, the mesh parameter is
  $h=\max_{K\in \Ch} h_K $. 
The set of the interior edges of the mesh is denoted by $\Eh$ and
the set of the boundary edges by $\Gh$. By $h_E$ we denote the length of the edge $E\in \Eh\cup\Gh$  
and by $h_i = \max_{K \in \Ch, c_i \in K} h_K$ the local mesh length around the corner $c_i$.

At times, we write in the estimates $a \lesssim b$ (or $a \gtrsim b$) when $a \leq Cb$ (or $a\geq Cb$), for some positive constant $C$, independent of the mesh parameter $h$ and the parameters $\tj, \kj, \eci$.  Moreover, we use the standard notation $(\cdot,\cdot)_R$ for the $L^2(R)$-inner product and write $(\cdot,\cdot)$ for the $L^2(\Omega)$-inner product. 

The (conforming) finite element space is defined as
\begin{equation}
V_h=\{ v\in H^2(\Omega) : v\vert_K \in V_K \ \forall K \in \Ch \}
\end{equation}
with the polynomial $V_K$  space satisfying
\begin{equation}P_p(K)\subset V_K \subset P_l(K),
\end{equation}
for some $p$ and $l$; $P_l(K)$ is the complete space of polynomials of degree $l$ in $K$. Examples of such spaces include (cf. \cite{ciarlet}), the Argyris triangle with $p=l=5$, the Bell triangle with $p=3$ and $l=5$ and the Bogner--Fox--Schmit rectangular element with $p=3$ and $l=6$.  The Hsieh--Clough--Tocher element is another option, but will lead to an additional term in the a posteriori estimator and hence is not included in the analysis.

The starting point for the design of the Nitsche method is the integration by parts formula \eqref{intbyparts}. From this we conclude that the exact solution $u$ satisfies the equation
\begin{equation}\label{exactsat}
\begin{aligned}
  &\int_{\Omega} \Mom(u) : \Kappa(v) \,\mathrm{d}x  \\
  &+ \sum_{i=1}^m \bigg\lbrace \int_{\Gamma_i} \mnn(u) \dn{v}\,\mathrm{d}s  -  \int_{\Gamma_i} \Vn(u)v\,\mathrm{d}s   -
 \llbracket \mns(u) \rrbracket|_{c_i} v(c_i)\bigg\rbrace
 \\& = \int_\Omega fv \,\mathrm{d}x \quad \forall v \in V_h .
 \end{aligned}
\end{equation}
Defining the bilinear form $\B(u,v)$ as the left-hand side in \eqref{exactsat}, it follows that
\begin{equation}
\B(u,v)=(f,v) \quad \forall  v\in V_h.
\end{equation}

 Next, we introduce the stabilizing and symmetrizing terms that will be added to the bilinear form. The spring constants and the loads corresponding to an edge
 $E \subset \Gamma_i$ are denoted by
  \begin{equation}
  \tj\vert_E=\tjE,\quad \kj\vert_E =\kjE, \quad \gvi\vert_E =\gvE, \quad \gmi\vert_E = \gmE.
 \end{equation}
 The first boundary condition in \eqref{eq:boundaryconds}, which at edge $E$ can be written  as
 \begin{equation}
 \label{altbcshear}
  \tjE \Vn(u) +  u = \tjE \gvE,
 \end{equation}
 thus prompts the definition of the residual
  \begin{equation}
  \Rv(v)=   \tjE (\Vn(v)-\gvE) +  v .
  \end{equation}
  Now, let $\gamma>0$ denote the stabilization parameter. The boundary
  condition \eqref{altbcshear} implies that
 \begin{equation}\label{firststab}
  \sum_{E \in \Gh}  \frac{1}{ \tjE + \gamma h_E^3} (\Rv(u), v)_E=0 \qquad \forall  v\in V_h,
  \end{equation}
  and
   \begin{equation}
  \sum_{E \in \Gh}  \frac{\gamma h_E^3}{ \tjE + \gamma h_E^3}\big (\Rv(u), \Vn(v)\big )_E=0 \qquad \forall  v\in V_h .
  \end{equation}
Similarly, we introduce the residuals for the remaining  boundary conditions, namely
\begin{equation}
\Rr(v)=    \kjE(\mnn(v)-\gmE )- \dn{v},  \quad \Rc(v) = \eci (\llbracket \mns(v) \rrbracket |_{c_i}- \gci)  + v(c_i),
\end{equation} 
and write them all together as
\begin{equation}\label{res=0}
\R_h(u,v)=0,
\end{equation}
where
\begin{equation}
\begin{aligned}
\R_h(u,v)
=& \sum_{E\in \Gh} \bigg\{       \frac{1}{ \tjE + \gamma h_E^3} (\Rv(u), v)_E +    \frac{\gamma h_E^3}{ \tjE + \gamma h_E^3}\big (\Rv(u), \Vn(v)\big )_E
\\
&\hspace{1cm}
-    
      \frac{1}{ \kjE + \gamma h_E} \left(\Rr(u),\dnt{v}\right)_E -    \frac{\gamma h_E}{ \kjE + \gamma h_E}\big (\Rr(u), \mnn(v)\big )_E  \bigg\}\hspace{-0.2cm}
\\
&  + \sum_{i=1}^m   \bigg\{ \ \frac{1}{ \eci + \gamma h_i^2}\Rc(u) v(c_i)
  +    \frac{\gamma h_i^2}{ \eci + \gamma h_i^2}\Rc(u)\llbracket \mns(v)\rrbracket |_{c_i} \bigg\}.
\end{aligned}
\end{equation}
Hence, the exact solution $u\in H^2(\Omega)$ satisfies 
 \begin{equation}
     \label{nitscheweakpre}
 \B(u,v)+\R_h(u,v)=(f,v) \quad \forall  v\in V_h.
 \end{equation}

  Finally, rearranging the terms, \eqref{nitscheweakpre} can be written as 
\begin{equation}
    \Bh(u, v) = \Lh(v) \quad \forall v \in V_h,
  \end{equation}
with the symmetric bilinear form $\Bh$ and the linear form $\Lh$ defined as
\begin{equation}
  \begin{aligned}
    \Bh(w, v) &= a(w,v)  + b_h(w, v) + c_h(w, v) + d_h(w, v),
     \\
    \Lh(v) &=l(v)+ f_h(v) + g_h(v) + l_h(v),
  \end{aligned}
\end{equation}
 where
 \begin{equation}
a(w,v) = \int_\Omega \Mom(w) : \Kappa(v) \,\mathrm{d}x,  \quad l(v)= \int_\Omega fv \,\mathrm{d}x, 
\end{equation}

\begin{equation}
\begin{aligned}
b_h(w,v)=   \sum_{E \in \Gh} \frac{1}{ \tjE + \gamma h_E^3} \Big\{&\gamma h_E^3\big(( \Vn(w), v)_E + (w, \Vn(v))_E \big) 
\\
&-\gamma h_E^3  \tjE\big(\Vn(w), \Vn(v)\big)_E +(w, v)_E \Big\},
\end{aligned}
\end{equation}

\begin{equation}
\begin{aligned}
c_h(w,v) = 
\sum_{E \in \Gh} 
 \frac{1}{ \kjE + \gamma h_E}
 \Big\{ &\gamma h_E  \left(\left(\mnn(w), \dnt{v}\right)_E + \left(\dnt{w}, \mnn(v)\right)_E \right)
  \\
&-\gamma  h_E  \kjE (\mnn(w), \mnn(v))_E +\left(\dnt{w}, \dnt{v}\right)_E  \Big\},
\end{aligned}
\end{equation}
\begin{equation}
\begin{aligned}
d_h(w,v)=   \sum_{i=1}^m  \frac{1}{\eci + \gamma h_i^2}
\Big\{&-\gamma h_i^2(\llbracket \mns(w) \rrbracket \vert_{c_i} v(c_i)  +
\llbracket \mns(v) \rrbracket \vert_{c_i} w(c_i))
  \\
  &
  -\gamma h_i ^2\eci  \llbracket \mns(w) \rrbracket \vert_{c_i}  \llbracket \mns(v) \rrbracket \vert_{c_i}  +w(c_i) v(c_i) \Big\},
\end{aligned}
\end{equation}
and 
\begin{equation}
\begin{aligned}
f_h(v)=  \sum_{E \in \Gh}\frac{ \tjE}{ \tjE + \gamma h_E^3} \Big\{ (\gvE, v)_E -\gamma   h_E^3(\gvE, \Vn(v))_E \Big\},
\end{aligned}
\end{equation}
\begin{equation}
\begin{aligned}
g_h(v) =  \sum_{E \in \Gh}\frac{\gamma  \kjE h_E}{ \kjE + \gamma h_E} \Big\{ -    \left(\gmE, \dnt{v}\right)_E - \gamma h_E(\gmE, \mnn(v))_E \Big\},
\end{aligned}
\end{equation}
\begin{equation}
\begin{aligned}
l_h(v) =  \sum_{i=1}^m
\frac{\eci}{\eci + \gamma h_i^2} \Big\{   - g_i v(c_i) -\gamma h_i^2   g_i \llbracket \mns(v) \rrbracket|_{c_i} \Big\}.
\end{aligned}
\end{equation}
  The 
  {\bf Nitsche method} now reads as follows: find $u_h \in V_h$
satisfying
\begin{equation}
  \label{weak}
  \Bh(u_h, v) = \Lh(v) \quad \forall v \in V_h.
\end{equation}

In the literature, it is often stated that the Nitsche's method and stabilized methods are consistent only for a sufficiently smooth solution. In the present case, the assumption would mean that $\Vn(u)\vert_E$ and $\mnn(u)\vert_E $ are in $L^2(E)$. However, recalling that we arrived at the method by adding weighted residuals to the variational formulation,  these residuals are smooth and vanish identically for the exact solution. Hence the following theorem holds.
\begin{thm}[Consistency]
The solution $u$ to \eqref{heikko}  satisfies
  \begin{equation}
    \Bh(u, v) = \Lh(v) \quad \forall v \in V_h.
  \end{equation}
\end{thm}

\section{Stability and a priori error analysis}

\label{sec:apriori}

The error analysis will be presented in the mesh-dependent norms
\begin{equation}
  \begin{aligned}
    \| w \|_h^2 = a(w,w) &+ \sum_{E \in \Gh} \bigg\{ \frac{1}{ \tjE + h_E^3} \| w \|_{0,E}^2 + \frac{1}{ \kjE + h_E} \left\| \dn{w} \right\|^2_{0,E} \bigg\}\\&+\sum_{i=1}^m \frac{1}{\eci+h_i^2} w(c_i)^2, \\
    \enorm{ w }_h^2 = \| w \|^2_h &+ \sum_{E \in \Gh} \Big\{ h_E^3 \| \Vn(w) \|_{0,E}^2 + h_E \| \mnn(w) \|_{0,E}^2 \Big\}
       \\ &+ \sum_{i=1}^m h_i^2 \big( \llbracket \mns(w) \rrbracket|_{c_i}  \big)^2. 
  \end{aligned}
\end{equation}
The following inverse estimate---true for every $v \in V_h$ with a constant $C_I>0$  independent of the   parameters $h, \tj, \kj, \eci$---can be proven by a scaling argument:
\begin{equation}
  \label{inverse}
  \begin{aligned}
  \sum_{E \in \Gh} \Big\{h_E^3 \| \Vn(v) \|_{0,E}^2 
  &+ h_E \| \mnn(v) \|_{0,E}^2\Big\} 
  +\sum_{i=1}^m h_i^2 \big( \llbracket \mns(w) \rrbracket|_{c_i}  \big)^2 \leq C_I   a(v, v).\hspace{-0.3cm}
  \end{aligned}
\end{equation}
Consequently,  the norms $\|\cdot\|_h$ and $\enorm{\,\cdot\,}_h$ are equivalent.

We will start by showing that the discrete bilinear form $\Bh$ is coercive. 
\begin{thm}[Stability]
  \label{thm:stab}
  Suppose that $0 < 2 \gamma < C_I^{-1}$. Then
  \begin{equation}
    \Bh(v, v) \gtrsim \|v\|_h^2 \quad \forall v \in V_h.
  \end{equation}
\end{thm}

\begin{proof}
 For $v \in V_h$, the Schwarz and Young's inequalities with some $\delta>0$ give
   \begin{equation}
   \begin{aligned}
    b_h(v,v)
    &=\sum_{E \in \Gh}\frac{1}{ \tjE + \gamma h_E^3} \Big\{-2\gamma h_E^3 ( \Vn(v), v)_E   
 -\gamma  \tjE h_E^3\Vert \Vn(v)\Vert_{0,_E}^2 +\Vert v\Vert_{0,E}^2 \Big\}\hspace{-0.2cm}
 \\  
 &\geq 
 \sum_{E \in \Gh}
 \frac{\gamma h_E^3}{ \tjE + \gamma h_E^3} \Big\{-2\gamma h_E^3 \Vert \Vn(v)\Vert_{0,E} \Vert v\Vert_{0,E}   \\
 &\hspace{3.07cm}-\gamma  \tjE h_E^3\Vert \Vn(v)\Vert_{0,_E}^2 +\Vert v\Vert_{0,E}^2 \Big\}
 \\
 &\geq 
  \sum_{E \in \Gh} 
  \frac{\gamma h_E^3}{ \tjE + \gamma h_E^3}\Big\{   
 -\gamma h_E^3 ( \tjE   +\delta \gamma h_E^3) \Vert \Vn(v)\Vert_{0,_E}^2 +(1-\delta^{-1}) \Vert v\Vert_{0,E}^2 \Big\}.\hspace{-0.2cm}
 \end{aligned}
  \end{equation}
  Choosing $\delta=2$ yields
  \begin{equation}
  \begin{aligned}
    b_h(v,v)&\geq
    \sum_{E \in \Gh} \bigg\{   
 -\gamma h_E^3 \frac{  \tjE   +2\gamma h_E^3}{ \tjE + \gamma h_E^3} \Vert \Vn(v)\Vert_{0,_E}^2 +  \frac{1}{ 2(\tjE + \gamma h_E^3)}\Vert v\Vert_{0,E}^2 \bigg\}
 \\
 & \geq  \sum_{E \in \Gh} \bigg\{   
 -2\gamma  h_E^3\Vert \Vn(v)\Vert_{0,_E}^2 + \frac{1}{ 2(\tjE + \gamma h_E^3)} \Vert v\Vert_{0,E}^2 \bigg\}
 .
 \end{aligned}
  \end{equation}
  By similar arguments,
  we get
  \begin{equation}
  c_h(v,v) \geq
\sum_{E \in \Gh}\! \Bigg\{  
 -2\gamma h_E \Vert \mnn(v)\Vert_{0,E}^2 +\frac{1}{ 2(\kjE + \gamma h_E)}\left\Vert \dn{v}\right\Vert_{0,E}^2\!\Bigg\},
  \end{equation}
  and
  \begin{equation}
  d_h(v,v)\geq    \sum_{i=1}^m  \bigg\{ -2\gamma h_i^2 \big( \llbracket \mns(v) \rrbracket|_{c_i}  \big)^2
   +\frac{1}{2(\eci + \gamma h_i^2)}  v(c_i)^2 \bigg\}. 
  \end{equation}
   This gives
   \begin{equation}
   \begin{aligned}
   \Bh(v,v) \geq a(v,v)
    &-2\gamma\bigg(\sum_{E \in \Gh} \Big\{h_E^3 \| \Vn(v) \|_{0,E}^2 
    + h_E \| \mnn(v) \|_{0,E}^2\Big\}\\
   &\hspace{1cm}+\sum_{i=1}^m h_i^2 \big( \llbracket \mns(w) \rrbracket|_{c_i}  \big)^2\bigg) 
\\& +\frac{1}{2} \bigg(
   \sum_{E \in \Gh} \bigg\{ \frac{1}{ \tjE + h_E^3} \| v \|_{0,E}^2 + \frac{1}{ \kjE + h_E} \left\| \dn{v} \right\|^2_{0,E} \bigg\}\\
 &\hspace{0.9cm}+\sum_{i=1}^m \frac{1}{\eci+h_i^2} v(c_i)^2\bigg).
   \end{aligned}
   \end{equation}
   The assertion is thus proved after choosing $0< \gamma < C_I^{-1}/2.$ \end{proof}

  Stability, consistency and the continuity of the bilinear form $\Bh$ together imply that
  \begin{equation}
    \label{eq:cea}
      \|u - u_h\|_h\lesssim   \enorm{ u-v }_h \quad \forall  v\in V_h.
  \end{equation}
  Using standard interpolation theory, we thus arrive at the following
  error estimate:
\begin{thm}[A priori estimate]
  \label{thm:apriori}
 Let $7/2 < s \leq p + 1$. For any solution $u \in H^s(\Omega)$ of \eqref{heikko}  it holds that
  \begin{equation}
    \|u - u_h\|_h \lesssim h^{s-2} \|u\|_s.
  \end{equation}
\end{thm}
\begin{rem} The regularity assumption $s> 7/2$ stems from the use of the mesh-dependent norm $\enorm{\,\cdot\,}_h$. When  Nitsche's method is applied to the Poisson problem, the corresponding assumption can be avoided, cf. \cite{LJS}. Similar approach could probably be used for the plate problem as well, but it is bound to be very technical and we did not attempt to carry it out. However, numerical computations with less regular solutions lead to optimal convergence rates also if $ s\leq 7/2$. 
\end{rem}
\section{A posteriori error analysis}

\label{sec:aposteriori}

The local error estimators are defined through
\begin{alignat}{2}\label{estimators}
    \eta_K^2(v) &= h_K^4 \| D \Delta^2 v - f \|_{0,K}^2\quad &&\forall K \in \Ch,\\
    \eta_{V,E}^2(v) &= h_E^{3} \| \llbracket \Vn(v) \rrbracket \|_{0,E}^2\quad &&\forall E \in \Eh,\\
    \eta_{M,E}^2(v) &= h_E \| \llbracket \mnn(v) \rrbracket \|_{0,E}^2\quad &&\forall E \in \Eh,\\
    \eta_{v,E}^2(v) &= \frac{h^{3}_E}{( \tjE + h_E^3)^2} \left\|  R^v_E(v)  \right\|_{0,E}^2 \quad &&\forall E \in \Gh,\\
    \eta_{r,E}^2(v) &= \frac{h_E}{( \kjE + h_E)^2} \left\| R^r_E(v)  \right\|_{0,E}^2 \quad &&\forall E \in \Gh,
    \\ \eta_{c,i} ^2(v) &= \frac{h_i^2}{(\eci+h_i^2)^2}    \big(   R^c_i (v)  \big)^2  &&i=1,\dots, m,
\end{alignat}
for any $v \in V_h$,
and the global error estimator $\eta_h$ reads as
\begin{equation}\begin{aligned}
  \displaystyle \eta_h^2(u_h) = &\sum_{K\in \mathcal{C}_h} \eta_K^2(u_h) +   \sum_{E\in \mathcal{E}_h} (\eta_{M,E}^2(u_h) + \eta_{V,E}^2(u_h)) \\
  &+ \sum_{E\in \Gh}  (\eta_{v,E}^2(u_h) + \eta_{r,E}^2(u_h))+\sum_{i=1}^m \eta_i(u_h)^2.
\end{aligned}
\end{equation}
In order to prove the reliability of the error estimator, we will use the following  assumption, justified by the a priori estimate for a regular enough solution.
\begin{ass}[Saturation assumption] There exists $0 <\beta <1$ such that
\begin{equation}
\|u-u_{h/2}\|_{h/2} \leq \beta  \|u-u_{h}\|_{h} ,
\end{equation}
    where $u_{h/2}\in V_{h/2}$ is the solution on the mesh $\mathcal{C}_{h/2}$ obtained by splitting the elements of the mesh  $\mathcal{C}_{h}$.
    \label{satu}
\end{ass}

\begin{thm}[Reliability]
If Assumption \ref{satu} holds true, then we have the estimate
\begin{equation}
\|u-u_h\|_{h} \lesssim \eta_h(u_h) \, .
\label{upperbound}
\end{equation}
\label{thm:aposteriori}
\end{thm}
\begin{proof}
From the coercivity of the bilinear form $ \A_{h/2}$ and the saturation assumption, it follows that
\begin{equation}
\|u-u_h\|_h\leq \frac{1}{1-\beta}  \|u_{h/2}-u_{h}\|_h \lesssim  \A_{h/2}(u_{h/2}-u_h,v)\, ,
\end{equation}
for some $v\in V_{h/2}$ such that $\|v\|_{h/2} =1$. 
Let  $\tilde{v}\in V_h$ be  the Hermite interpolant of $v\in V_{h/2}$.  We have the following estimates
\begin{equation}
\begin{aligned}
& \sum_{K\in \mathcal{C}_{h/2}} h_K^{-4} \| v-\tilde{v} \|_{0,K}^2
+  \sum_{E\in \mathcal{G}_{h} \cup \mathcal{E}_{h/2} } \Big\{ h_E^{-1}  \|\nabla(v-\tilde{v})\|_{0,E}^2 +
h_E^{-3}  \|v-\tilde{v}\|_{0,E}^2 \Big\} \\
& + \sum_{E\in \mathcal{G}_{h/2} }  \bigg\{  h_E^3\|V_n(v-\tilde{v})\|_{0,E}^2+ h_E\|M_{nn}(v-\tilde{v})\|_{0,E}^2
 \\ 
& \hspace{1.7cm} + \frac{1}{ \tjE+h_E^3} \|v-\tilde{v}\|_{0,E}^2 + \frac{1}{ \tjE+h_E} \left\|\dn{(v-\tilde{v})}\right\|_{0,E}^2\, \bigg\} \\
 &+\sum_{i=1}^m h_i^2 \big( \llbracket \mns(v-\tilde{v}) \rrbracket|_{c_i}  \big)^2 
\\ & \leq C\, \|v\|_{h/2}^2 \lesssim 1 ,
\end{aligned} 
\label{hermite}
\end{equation}
and 
\begin{equation}
\|\tilde{v}\|_{h/2}\lesssim \|  v   \|_{h/2} \lesssim \|v\|_{h} \lesssim 1 .
\end{equation}
Let $w=v-\tilde{v}$ and write
\begin{equation}   
  \A_{h/2}(u_{h/2}-u_{h},v) =   \A_{h/2}(u_{h/2}-u_h,w) +   \A_{h/2}(u_{h/2}-u_h,\tilde{v}).\label{estim1}\end{equation}
To estimate the first term in \eqref{estim1}, we write it as
\begin{equation}\label{A1}
\begin{aligned}
   \A_{h/2}(u_{h/2}-u_h,w)& = 
    \A_{h/2}(u_{h/2},w)-   \A_{h/2}(u_h,w)
 \\& = (f,w)  -\Bf(u_h,w)-\R_{h/2} (u_h, w)
 .
\end{aligned}
\end{equation}
A repeated partial integration,  and the fact that $w$ vanishes at the nodes of $\Ch$ gives  \begin{equation}\label{A2}
\begin{aligned}
&(f,w) - \Bf(u_h,w) \\
&= \sum_{K\in \mathcal{C}_{h}} \Big\{(f-D\Delta^2 u_h,w)_{K} -( V_n(u_h),w)_{\partial K} + \left(M_{nn}(u_h), \dnt{w}\right)_{\partial K}\Big\} \\ 
&=  \sum_{K\in \mathcal{C}_{h}} (f-D\Delta^2 u_h,w)_{K}  \\ 
& \phantom{=}\,+  \sum_{E\in \mathcal{E}_{h}}\Big\{ - ( \llbracket \Vn(u_h) \rrbracket , \,w)_{E} +\left( \llbracket M_{nn}(u_h)\rrbracket , \dnt{w}\right)_{E}\Big\} .
\end{aligned}
\end{equation}
Recalling that $w=v-\tilde v$,   estimate \eqref{hermite} leads to the bounds
\begin{equation}
\begin{aligned}
&\sum_{K\in \mathcal{C}_{h}}  (f-D\Delta^2 u_h,w)_{K}    
\\ 
& \leq \left(  \sum_{K\in \mathcal{C}_{h}} h_K^4 \| D \Delta^2 u_h - f \|_{0,K}^2\right)^{1/2} \left( \sum_{K\in \Ch} h_K^{-4} \Vert w \Vert_{0,K}^2\right)^{1/2}
 \\ &  \lesssim \left(  \sum_{K\in \mathcal{C}_{h}} h_K^4 \| D \Delta^2 u_h - f \|_{0,K}^2\right)^{1/2} \lesssim \eta_h(u_h) ,
 \end{aligned}
\end{equation}
and 
\begin{equation}
\begin{aligned}
  &\sum_{E\in \mathcal{E}_{h}}   \Big\{  ( \llbracket \Vn(u_h) \rrbracket , \,w)_{E}  +\left( \llbracket M_{nn}(u_h)\rrbracket , \dnt{w}\right)_{E}\Big\} 
  \\
  &\leq \left( \sum_{E\in \mathcal{E}_{h}} h_E^{3} \| \llbracket \Vn(u_h) \rrbracket \|_{0,E}^2\right)^{1/2}
    \left( \sum_{E\in \mathcal{E}_{h}} h_E^{-3} \|  w  \|_{0,E}^2\right)^{1/2}
    \\
    &\phantom{=}~+  
  \left( \sum_{E\in \mathcal{E}_{h}} h_E \| \llbracket \mnn(u_h) \rrbracket \|_{0,E}^2 \right)^{1/2} 
  \left( \sum_{E\in \mathcal{E}_{h}} h_E^{-1}\left\| \dn{w}  \right\|_{0,E}^2\right)^{1/2}
  \\
  &\lesssim \left( \sum_{E\in \mathcal{E}_{h}} h_E^{3} \| \llbracket \Vn(u_h) \rrbracket \|_{0,E}^2\right)^{1/2}+  
  \left( \sum_{E\in \mathcal{E}_{h}} h_E \| \llbracket \mnn(u_h) \rrbracket \|_{0,E}^2 \right)^{1/2}
  \\  &\lesssim \eta_h(u_h) .
\end{aligned}
\end{equation}
Moreover, using the Schwarz inequality on each $E \in \mathcal{G}_{h/2}$, the Cauchy inequality for sums, and estimate \eqref{hermite}, we get
\begin{equation}
-\R_{h/2}(u_h,w) \lesssim \eta_{h/2}(u_h) \lesssim \eta_h(u_h).
\end{equation}

Next, we consider the second  term in \eqref{estim1}. First, we note that   
\begin{equation}
 \begin{aligned}
     \A_{h/2}(u_{h/2}-u_h,\tilde{v})&=    \A_{h/2}(u_{h/2},\tilde{v})-   \A_{h/2}(u_h,\tilde{v})
    \\
    &= \R_{h}(u_h, \tilde v)-\R_{h/2}(u_h, \tilde v).
  \end{aligned}  
  \end{equation}
  For an edge $E \in \mathcal{G}_{h/2}$ such that $E\subset F$, with $F\in \mathcal{G}_h$, it holds 
$h_F=2h_E$. Thus we get
\begin{equation}
\begin{aligned}
&\R_{h}(u_h, \tilde v)-\R_{h/2}(u_h, \tilde v)\\
&= 
  \sum_{E\in \mathcal{G}_{h/2}} \bigg\{    -  \frac{7 \gamma h_E^3}{( \tjE + \gamma h_E^3)( \tjE + 8\gamma   h_E^3)} (\Rv(u_h), \tilde v)_E 
\\&\hspace{1.8cm}
+   \frac{7\tjE h_E^3}{ ( \tjE + \gamma h_E^3)( \tjE + 8\gamma  h_E^3)}\big(\Rv(u_h), \Vn(\tilde v)\big )_E
\\
&
\hspace{1.8cm}+    
    \frac{\gamma h_E}{ (\kjE + \gamma h_E)(\kjE +2 \gamma h_E)} \big(\Rr(u_h),\dnt{\tilde v}\big)_E
 \\
 &\hspace{1.8cm}    -   \frac{\gamma \kjE h_E}{(\kjE + \gamma h_E)(\kjE +2 \gamma h_E)}\big(\Rr(u_h), \mnn(\tilde v)\big )_E  \bigg\}
\\ 
  &\phantom{=}~+\sum_{i=1}^m \bigg\{  -\frac{3\gamma h_i^2}{( \eci + \gamma h_i^2)( \eci +4 \gamma h_i^2)}\Rc(u_h) v(c_i)
 \\
 & \hspace{1.65cm} +   \frac{-3\gamma \eci  h_i^2}{( \eci + \gamma h_i^2)( \eci +4 \gamma h_i^2)}\Rc(u)\llbracket \mns(\tilde v)\rrbracket |_{c_i} \bigg\}.
\end{aligned}
\end{equation}
The first term above we estimate as follows:
\begin{equation}
\begin{aligned}
&\Bigg\vert \sum_{E\in \mathcal{G}_{h/2}}      \frac{7 \gamma h_E^3}{( \tjE + \gamma h_E^3)( \tjE + 8\gamma   h_E^3)} (\Rv(u_h), \tilde v)_E\Bigg\vert
\\
& \lesssim
 \sum_{E\in \mathcal{G}_{h/2}}       \frac{  h_E^3}{( \tjE +   h_E^3)^2} \Vert \Rv(u_h)\Vert_{0,E} \Vert  \tilde v\Vert_{0,E}
 \\
 & \lesssim \left( \sum_{E\in \mathcal{G}_{h/2}}       \frac{  h_E^3}{( \tjE +   h_E^3)^2} \Vert \Rv(u_h)\Vert_{0,E}^2\right)^{1/2}
  \left(\sum_{E\in \mathcal{G}_{h/2}}       \frac{  h_E^3}{( \tjE +   h_E^3)^2}  \Vert  \tilde v\Vert_{0,E}^2\right)^{1/2}
  \\
  & \lesssim \left( \sum_{E\in \mathcal{G}_{h/2}}       \frac{  h_E^3}{( \tjE +   h_E^3)^2} \Vert \Rv(u_h)\Vert_{0,E}^2\right)^{1/2}
  \left(\sum_{E\in \mathcal{G}_{h/2}}       \frac{ 1}{( \tjE +   h_E^3)}  \Vert  \tilde v\Vert_{0,E}^2\right)^{1/2}
  \\
  &\lesssim \eta_{h/2} (u_h) \|\tilde{v}\|_{h/2} \lesssim \eta_h(u_h).
\end{aligned}
\end{equation}
The other terms are estimated in the same way.
Now, estimating separately each term above, we conclude that
\begin{equation}
\R_{h}(u_h, \tilde v)-\R_{h/2}(u_h, \tilde v)   \lesssim \eta_h(u_h).
\end{equation}
The claim is now proved by collecting the estimates.
  \end{proof}

Next we turn to the lower bounds.
We denote by $\omega_E$ the union of two elements
that have $E \in \Eh$ as one of their edges, and by $K(E)$ the
element which has $E \in \Gh$ as one of its edges.
The data oscillations are defined as
\begin{align*}
  \osc_K(f) &= h_K^2 \| f - \fh \|_{0,K}, \\
  \osc_{v,E}(\gvE) &= \frac{h_E^{3/2}}{ \tjE + h_E^3} \|  \tjE (\gvE - \gvh)    \|_{0,E}, \\
  \osc_{r,E}(\gmE ) &= \frac{h_E^{1/2}}{ \kjE + h_E} \|  \kjE (\gmE - \grh)   \|_{0,E},
\end{align*}
where $\fh, \, \gvh, \, \grh$ are polynomial approximations to $f, \, \gvE$ and $\gmE$, respectively.

\begin{thm}[Efficiency]
  For all $v \in V_h$ it holds
  \begin{alignat}{2}
    \label{eq:lowboundint} \eta_K(v) &\lesssim |u-v|_{2,K} + \osc_K(f) &&K\in \Ch, \\
    \label{eq:edgebound1} \eta_{V,E}(v) & \lesssim |u-v|_{2,E} + \sum_{K \subset \omega_E} \osc_K(f) && E\in \Eh, \\
    \label{eq:edgebound2}  \eta_{M,E}(v) &\lesssim |u-v|_{2,\omega_E} + \sum_{K \subset \omega_E} \osc_K(f) && E\in \Eh, \\
    \label{eq:edgebound3} \eta_{{v,E}}(v) &\lesssim |u-v|_{2,\omega_E} + \frac{1}{\sqrt{ \tjE + h_E^3}}\|u - v\|_{0,E}\\
    &\quad+ \osc_{K(E)}(f) +  \osc_{v,E}(\gvE) && E\in \Gh, \nonumber \\
    \label{eq:edgebound4} \eta_{r,E}(v) &\lesssim |u-v|_{2,\omega_E} + \frac{1}{\sqrt{ \kjE + h_E}}\left\|\dn{(u - v)}\right\|_{0,E}\quad\\
    &\quad+ \osc_{K(E)}(f) +  \osc_{r,E}(\gmE) && E\in \Gh.\nonumber
  \end{alignat}
  \label{thm:lower}
\end{thm}

\begin{proof}
  The bounds \eqref{eq:lowboundint}, \eqref{eq:edgebound1}, and \eqref{eq:edgebound2} are proved in~\cite{gustafsson-2018-poster-estim}. Let us now consider \eqref{eq:edgebound3}. The triangle inequality gives
\begin{equation}
\label{D0}
    \eta_{{v,E}}(v) \leq \frac{h_E^{3/2}}{ \tjE + h_E^3}\|R_{E,h}^v(v)\|_{0,E} + \osc_{v,E}(\gvE),
\end{equation}
  where
\begin{equation}
  R_{E,h} ^v(v) =  \tjE(\Vn(v) -\gvh) + v .
\end{equation}
  Let $\phi_E$ denote the eight degree polynomial with support in $K(E)$
satisfying
\begin{equation}
\begin{aligned}
  \frac{\partial\phi_E}{\partial \boldsymbol{n}}\Big|_{\partial K(E)} &= 0,
  \\
 \phi_E &> 0 \ \mbox{on } E \mbox{ and in the interior of } K,
 \\
 \phi_E&=0 \ \mbox{on } \partial K(E) \setminus E,
 \\ \max \phi_E&=1.
\end{aligned}
\end{equation}
Denoting $w = \phi_E R_{E,h}^v(v)$ we have
  \begin{equation}\label{D1} 
  \begin{aligned}
    \|R_{E,h}^v\|_{0,E}^2 &\lesssim \| \phi_E^{1/2} R_{E,h}^v \|_{0,E}^2 \\
    &= (R_{E,h}^v, w)_E =(R_{E}^v, w)_E+(\gvE-\gvh,w)_E
     .
  \end{aligned}
  \end{equation}
  Integrating by parts we have 
  \begin{equation}
  (\Vn(v),w)_E=- (D\Delta^2 v, w)_{K(E)}+ (\Mom(v),\Kappa(w))_{K(E)}.
  \end{equation}
  On the other hand, from \eqref{heikko} we get
  \begin{equation}
  (\gvE,w)_E+(f,w)_{K(E)}= (\Mom(u), \Kappa(w))_{K(E)} +\frac{1}{\tjE} (u,w)_{E},
  \end{equation}
  and, hence, it holds that
  \begin{equation}
  \begin{aligned}
(R_{E}^v, w)_{E}=&\,\tjE\big (  (\Mom(v-u),\Kappa(w))_{K(E)}  + (f-D\Delta^2 v, w)_{K(E)} \big)\\&+(v-u,w)_E.
\end{aligned}
  \end{equation}
  By scaling arguments, we have
  \begin{equation}
  \Vert \Kappa( w )\Vert_{K(E)} \lesssim h_E^{-3/2} \Vert   w  \Vert_{0,E}  \lesssim h_E^{-3/2} \Vert  R_{E,h}^v \Vert_{0,E},  
  \end{equation} 
  and
    \begin{equation}
  \Vert w \Vert_{K(E)} \lesssim h_E^{1/2} \Vert   w  \Vert_{0,E}  \lesssim h_E^{1/2} \Vert  R_{E,h}^v \Vert_{0,E}. 
  \end{equation} 
This implies that 
  \begin{equation}
  \begin{aligned}
 \vert (R_{E}^v, w)_{E}\vert \lesssim \Big(  &\tjE \big( h_E^{-3/2} |u - v|_{2,K(E)} + h^{1/2} \| D \Delta^2 v - f \|_{0,K(E)} \big)\\
 &+\Vert u-v \Vert_{0,E}        \Big) \Vert  R_{E,h}^v \Vert_{0,E}.
 \label{est541}
 \end{aligned}
   \end{equation}
 From \eqref{est541} and \eqref{D1} we finally conclude that
 \begin{equation}
 \begin{aligned}
& \frac{h_E^{3/2}}{\tjE+h_E^3}  \Vert R_{E,h}^v \Vert_{0,E}\\
 & \lesssim 
  \frac{\tjE}{\tjE+h_E^3} | u - v|_{2,K(E)} +  \frac{\tjE h_E^{2}}{\tjE+h_E^3}   \| D \Delta^2 v - f \|_{0,K(E)}
  \\ &\phantom{\lesssim}~ + \frac{h_E^{3/2}}{\tjE+h_E^3}  \Vert u-v \Vert_{0,E} + \osc_{v,E}(\gvE) 
  \\
  & 
  \lesssim 
  | u - v |_{2,K(E)} +   h_E^{2} \| D \Delta^2 v - f \|_{0,K(E)}  
  \\&\phantom{\lesssim}~ +  (\tjE+h_E^3)^{-1/2}   \Vert u-v \Vert_{0,E}+\osc_{v,E}(\gvE)  .
  \end{aligned}
   \label{est542}
   \end{equation}
Estimate \eqref{est542} together with \eqref{D0} and \eqref{eq:edgebound1} leads to the asserted  estimate \eqref{eq:edgebound3}.

The lower bound   \eqref{eq:edgebound4} is proved in an analogous manner using a weight function $\phi_E'$  satisfying
\begin{equation}
\begin{aligned}
  \frac{\partial\phi_E'}{\partial \boldsymbol{n}}\Big|_{E}& >0,
  \\
    \frac{\partial\phi_E'}{\partial \boldsymbol{n}}\Big|_{\partial K(E)\setminus E}&=0,
    \\
 \phi_E'&> 0 \  \mbox{in the interior of } K,
 \\
 \phi_E'&=0 \ \mbox{on } \partial K(E),
 \\ \max \phi_E' &=1.
\end{aligned}
\end{equation}
\\[0.1cm]\phantom{nothing}
\end{proof}

\begin{rem}
We are unable to prove the efficiency of the
corner estimators $\eta_{c,i}$ for all
values $0 \leq \kj, \tj, \eci \leq \infty$, $i=1,\dots,m$.
In particular,
when $\eci \neq 0$ and $\tj$ is close to zero
there seems to be a nontrivial
coupling between $\eta_{c,i}$ and $R^v_E$.
\end{rem}

\section{Computational results}

\label{sec:results}

For numerical experiments,
we implement a finite element solver based on the Argyris element.
Our solver allows enforcing boundary conditions either via the Nitsche
method of Section~\ref{sec:nitsche}
or, in simple cases, via the classical method of
directly eliminating degrees-of-freedom.
In all examples, we consider the square domain
$\Omega = [0,1]^2$ defined by the corner points
\[
c_1 = (0,0), \quad c_2 = (1,0), \quad c_3 = (1,1), \quad c_4 = (0,1).
\]

\subsection{Clamped square plate}

Let $E=1$, $\nu=0.3$, and $d=1$. The analytical solution to the
fully clamped problem ($\kj = \tj = \eci = 0,~i=1,\dots,4$)
with  loading
\begin{align*}
\gvi &= \gmi = \gci = 0,~i=1,\dots,4,\\
f(x,y)&=8 \pi^4 D(\cos^2 \pi x \cos^2 \pi y - 2 \sin^2 \pi x \cos^2 \pi y \\
&\qquad \qquad - 2 \cos^2 \pi x \sin^2 \pi y  + 3 \sin^2 \pi x \sin^2 \pi y),
\end{align*}
reads  as follows
\[
u(x,y)=\sin^2 \pi x \sin^2 \pi y.
\]
To validate our implementation, we solve the problem using a uniform mesh
family for both Nitsche's method with $\gamma=10^{-3}$ and the classical
method---the meshes and the solutions are given in Figure~\ref{fig:ex1}.  The
approximate deflections $u_h(1/2, 1/2)$ presented in Table~\ref{tab:ex1}
show how the exact maximum deflection $u(1/2, 1/2)=1$
is reproduced with high
accuracy by both approaches.

\begin{figure}
  \includegraphics[width=0.191\textwidth]{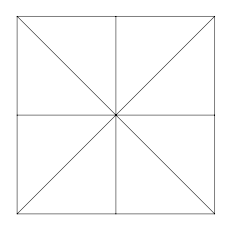}
  \hspace{0.59cm}
  \includegraphics[width=0.191\textwidth]{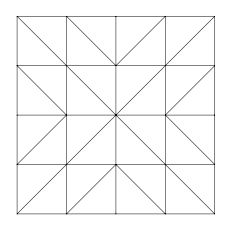}
  \hspace{0.59cm}
  \includegraphics[width=0.191\textwidth]{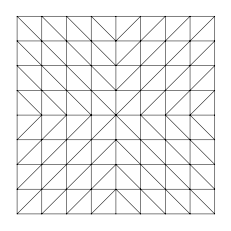}
  \hspace{0.59cm}
  \includegraphics[width=0.191\textwidth]{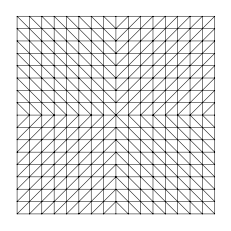}
  \\[0.5cm]
  \includegraphics[width=0.245\textwidth]{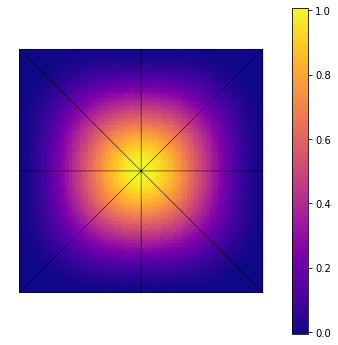}
  \includegraphics[width=0.245\textwidth]{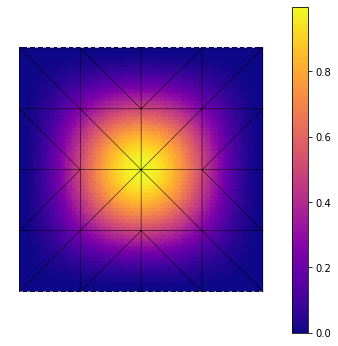}
  \includegraphics[width=0.245\textwidth]{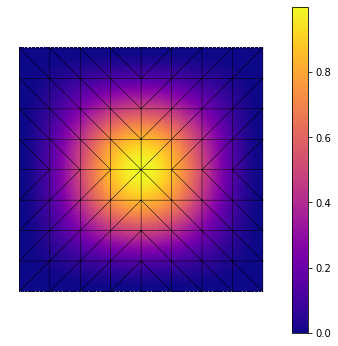}
  \includegraphics[width=0.245\textwidth]{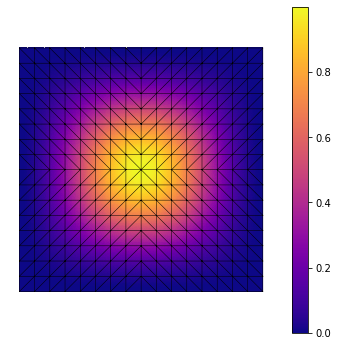}
  \\[0.2cm]
  \includegraphics[width=0.245\textwidth]{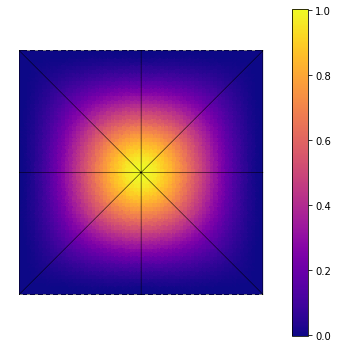}
  \includegraphics[width=0.245\textwidth]{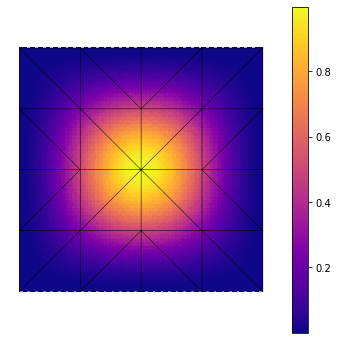}
  \includegraphics[width=0.245\textwidth]{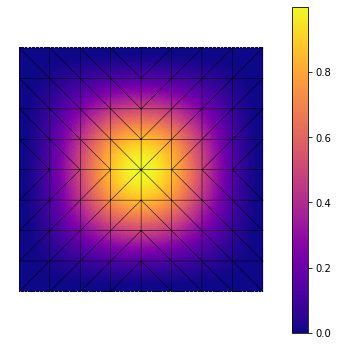}
  \includegraphics[width=0.245\textwidth]{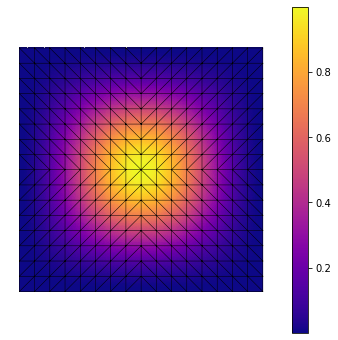}
  \caption{Mesh sequence (top row) and the deflections computed with the classical method (middle row) and
  Nitsche's method (bottom row).  The source code for reproducing
  these results is available in~\cite{tom_gustafsson_2020_3925365}.}
  \label{fig:ex1}
\end{figure}

\begin{table}
  \caption{Pointwise deflections in the mid point of the clamped square plate.}
  \label{tab:ex1}
  \centering
  \begin{filecontents*}{ex1/ex1results.csv}
h,nitsche,trad,estimator,nitscheh2,nitschemd
0.7071067811865476,1.0058541503536569,1.010907442937636,24.855283655336606,2.5088807551952255,2.5089000191977426
0.3535533905932738,0.99996169866189,1.000041958697418,2.344469781828662,0.19353107076160392,0.19353188364887086
0.1767766952966369,0.9999951147609446,0.9999951465252185,0.16108796784743407,0.013066860954957946,0.013066890305629355
0.08838834764831845,0.9999999128919681,0.9999998981696189,0.010316287744846418,0.0007650004779910054,0.0007650012151992773
\end{filecontents*}
\pgfplotstabletypeset[
    col sep=comma,
    sci zerofill={true},
    precision=7,
    columns={h,nitsche,trad},
    columns/h/.style={
     column name={$h$},
    },
    columns/nitsche/.style={
     column name={Nitsche, $u_h(1/2, 1/2)$},
    },
    columns/trad/.style={
     column name={traditional, $u_h(1/2, 1/2)$},
    },
  ]{ex1/ex1results.csv}
\end{table}

Continuing only with the Nitsche method, we calculate the discrete norm $\|u-u_h\|_h$
and the following elementwise a posteriori error indicator:
\begin{align*}
  E_K(u_h) = &\,h_K^2 \| D \Delta^2 u_h - f \|_{0,K} + \frac12 h_K^{3/2} \|
  \llbracket \Vn(u_h) \rrbracket \|_{0,\partial K} \\ &+ \frac12 h_K^{1/2} \|
  \llbracket \mnn(u_h) \rrbracket \|_{0,\partial K} + h^{3/2}_K \| u_h
  \|_{0,\partial K \cap \partial \Omega} \\
  &+ h^{1/2}_K \left\| \dn{u_h} \right\|_{0,\partial K \cap \partial \Omega}
   + h_K^{-1} \sum_{i = 1}^4 u_h(c_i) \chi_K(c_i),
\end{align*}
where
\begin{equation}
\chi_K(x) =
    \begin{cases}
    1 & \text{if $x \in K$},\\
    0 & \text{otherwise}.
    \end{cases}
\end{equation}
The results are summarized in
Table~\ref{tab:ex2}.  We observe that the convergence rates
are consistent with the expected rate $\mathcal{O}(h^4)$
for fifth degree polynomials and regular solutions.
Moreover, the error indicator converges
with similar rates as the true error
which is also a consequence of Theorems~\ref{thm:aposteriori}
and \ref{thm:lower}.

\begin{table}
  \caption{Convergence of the error and the error estimator.}
  \label{tab:ex2}
  \centering
  \pgfplotstabletypeset[
    create on use/rate/.style={
      create col/dyadic refinement rate={nitschemd},
    },
    create on use/rate2/.style={
      create col/dyadic refinement rate={estimator},
    },
    columns={h,nitschemd,rate,estimator,rate2},
    col sep=comma,
    sci zerofill={true},
    precision=7,
    columns/h/.style={
     column name={$h$},
    },
    columns/nitschemd/.style={
     column name={$\|u-u_h\|_h$},
    },
    columns/rate/.style={
      dec sep align,
    },
    columns/estimator/.style={
      column name={$\sqrt{\sum_{K \in \Ch} E_K^2(u_h)}$},
    },
    columns/rate2/.style={
      column name={rate},
      dec sep align,
    },
  ]{ex1/ex1results.csv}
\end{table}

\subsection{Plate supported at the corners}

Next we consider the same problem with loading $f=1$ and $\eci=0$,
$\kj = \tj = \infty$, $\gvi = \gmi = \gci = 0$, $i=1,\dots,4$,
i.e.~the deflection is prevented only at the corners of the plate.
We investigate the convergence rate
of the error indicator
\begin{align*}
  E_K(u_h) = &\,h_K^2 \| D \Delta^2 u_h - 1 \|_{0,K} + \frac12 h_K^{3/2} \|
  \llbracket \Vn(u_h) \rrbracket \|_{0,\partial K} \\ &+ \frac12 h_K^{1/2} \|
  \llbracket \mnn(u_h) \rrbracket \|_{0,\partial K} + h^{3/2}_K \| \Vn(u_h)
  \|_{0,\partial K \cap \partial \Omega} \\
  & + h^{1/2}_K \left\| \mnn(u_h) \right\|_{0,\partial K \cap \partial \Omega}
    + h_K^{-1} \sum_{i=1}^4 u_h(c_i) \chi_K(c_i)
\end{align*}
as a function of the number of degrees-of-freedom $N$ with uniform and adaptive
mesh refinement strategies.  The results shown in Figure~\ref{fig:aposteriori}
indicate that an adaptive refinement based on the error indicator $E_K(u_h)$ successfully recovers
the convergence rate $\mathcal{O}(N^{-2})$.

\pgfplotstableread[col sep=comma]{ex2/ex2uniform.csv}{\uniformtab}
\pgfplotstableread[col sep=comma]{ex2/ex2adaptive.csv}{\adaptivetab}

\begin{figure}
  \centering
  \includegraphics[width=0.24\textwidth]{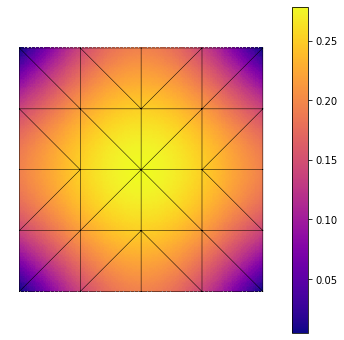}
  \includegraphics[width=0.24\textwidth]{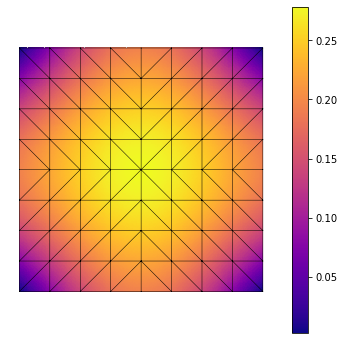}
  \includegraphics[width=0.24\textwidth]{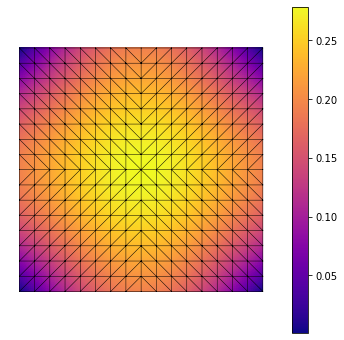}
  \includegraphics[width=0.24\textwidth]{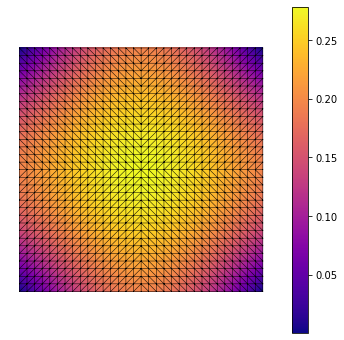}\\[0.5cm]
  \includegraphics[width=0.24\textwidth]{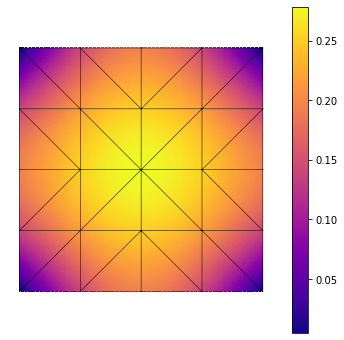}
  \includegraphics[width=0.24\textwidth]{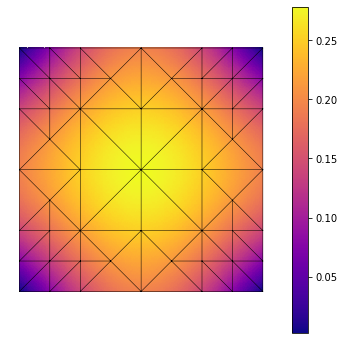}
  \includegraphics[width=0.24\textwidth]{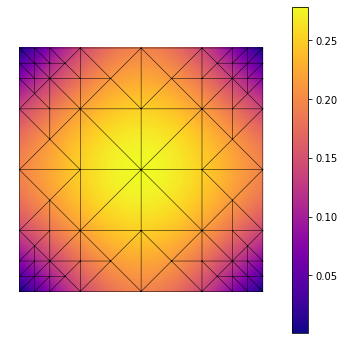}
  \includegraphics[width=0.24\textwidth]{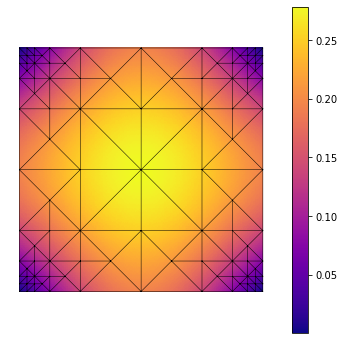}\\[0.5cm]
  \begin{tikzpicture}[scale=0.9]
    \begin{axis}[
        xmode = log,
        ymode = log,
        xlabel = {$N$},
        ylabel = {$\sqrt{\sum_{K \in \Ch} E_K^2(u_h)}$},
        grid = both,
      ]
      \addplot table[
        x=N,
        y=estimators] {\adaptivetab};
      \addplot table[
        x=N,
        y=estimators] {\uniformtab};
      \addplot+ [black, dashed, domain=7e2:2e3, mark=none] {exp(-1*ln(x) + ln(3e-1) - (-0.65)*ln(6e2)))} node[right,pos=1.0]{$\mathcal{O}(N^{-1})$};
       \addplot+ [black, dashed, domain=7e2:2e3, mark=none] {exp(-2*ln(x) + ln(3e-1) - (-1.5)*ln(6e2)))} node[right,pos=1.0]{$\mathcal{O}(N^{-2})$};

      \addlegendentry{Adaptive}
      \addlegendentry{Uniform}
    \end{axis}
  \end{tikzpicture}
  \caption{The first four meshes in the uniform (top row) and the adaptive (middle row) mesh sequences
    with the corresponding solutions and the a posteriori error estimator (bottom row) plotted as a
    function of the number of degrees-of-freedom $N$ with $\gamma = 10^{-3}$.
    The source code for reproducing the example is available in~\cite{zenodo2}.}
  \label{fig:aposteriori}
\end{figure}

\subsection{Elastic support with applied loads at the boundaries}

As the final example, we consider the square plate problem with $\nu=0$,
$\tj = 1$, $\eci = \infty$, the loading $f=0$, and
\begin{equation}
  \gvi = g^v(y) = \begin{cases}
    1 & \text{if $y < 3/4$}, \\
    0 & \text{otherwise},
  \end{cases} \qquad
  \gmi = g^r(y) = \begin{cases}
    10 & \text{if $y < 1/4$}, \\
    0 & \text{otherwise},
  \end{cases}
\end{equation}
for each $i=1,\dots,4$.
Our aim is to compare the adaptive meshes resulting from the Nitsche method
and the classical method when $\kj = \varepsilon^r = 10^{-k}$, $k=0,2,4,6$, $i=1,\dots,4$.  The error indicator for Nitsche's method
reads as
\begin{align*}
  E_K(u_h) = &\,h_K^2 \| D \Delta^2 u_h \|_{0,K} + \frac12 h_K^{3/2} \|
  \llbracket \Vn(u_h) \rrbracket \|_{0,\partial K} + \frac12 h_K^{1/2} \|
  \llbracket \mnn(u_h) \rrbracket \|_{0,\partial K} \\
  &+ \frac{h^{3/2}_K}{1 + h_K^3} \| \Vn(u_h) - g^v + u_h \|_{0,\partial K \cap \partial \Omega} \\
  &+ \frac{h_K^{1/2}}{\varepsilon^r + h_K} \left\| \varepsilon^r(\mnn(u_h) - g^r) - \dn{u_h} \right\|_{0,\partial K \cap \partial \Omega} \\
  &+ h_K \sum_{i = 1}^4 \llbracket \mns(u_h) \rrbracket |_{c_i} \chi_K(c_i).
\end{align*}
The error indicator for the classical method is
\begin{align*}
  E_K(u_h) = &\,h_K^2 \| D \Delta^2 u_h \|_{0,K} + \frac12 h_K^{3/2} \|
  \llbracket \Vn(u_h) \rrbracket \|_{0,\partial K} + \frac12 h_K^{1/2} \|
  \llbracket \mnn(u_h) \rrbracket \|_{0,\partial K} \\
  &+ h^{3/2}_K \| \Vn(u_h) - g^v + u_h \|_{0,\partial K \cap \partial \Omega} \\
  &+ h_K^{1/2} \left\| \mnn(u_h) - g^r - \frac{1}{\varepsilon^r}\dn{u_h} \right\|_{0,\partial K \cap \partial \Omega}.
\end{align*}
The resulting adaptive meshes are presented in Figure~\ref{fig:adaptcomp}.  The results show
that the classical method can lead to overrefinement in the case of stiff elastic supports.

\begin{figure}
  \centering
  $\varepsilon^r = 1$\\
  \includegraphics[width=0.31\textwidth]{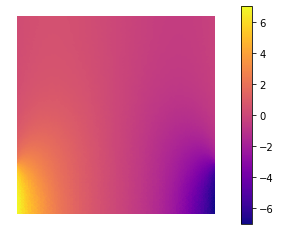}
  \hspace{0.01\textwidth}
  \includegraphics[width=0.25\textwidth]{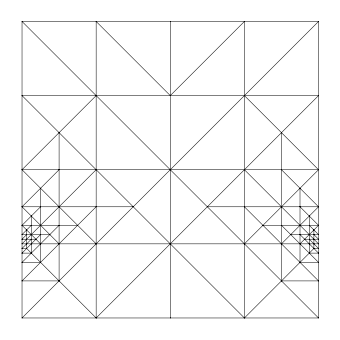}
  \includegraphics[width=0.25\textwidth]{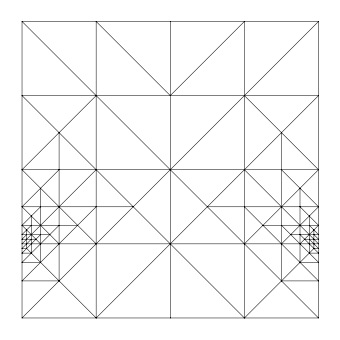}
  \hspace{-0.04\textwidth}
  \\
  $\varepsilon^r = 10^{-2}$\\
  \includegraphics[width=0.33\textwidth]{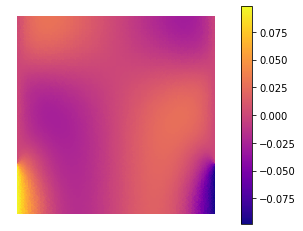}
  \includegraphics[width=0.25\textwidth]{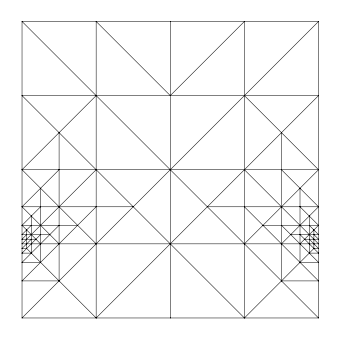}
  \includegraphics[width=0.25\textwidth]{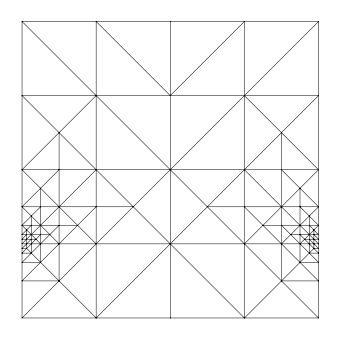}\\
  $\varepsilon^r  = 10^{-4}$\\
  \includegraphics[width=0.33\textwidth]{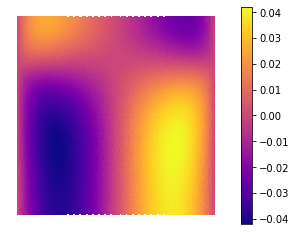}
  \includegraphics[width=0.25\textwidth]{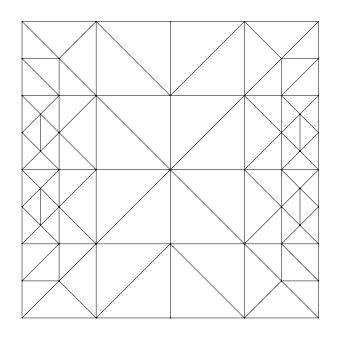}
  \includegraphics[width=0.25\textwidth]{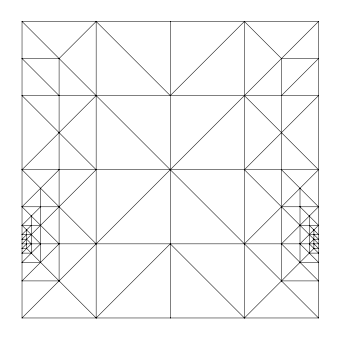}\\
  $\varepsilon^r  = 10^{-6}$\\
  \includegraphics[width=0.33\textwidth]{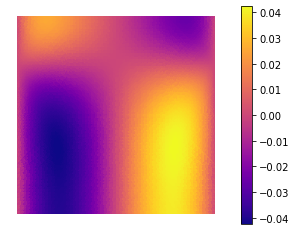}
  \includegraphics[width=0.25\textwidth]{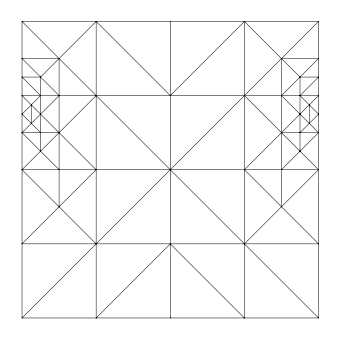}
  \includegraphics[width=0.25\textwidth]{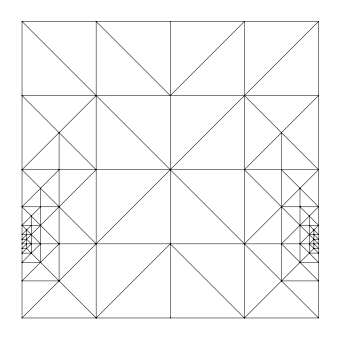}
  \caption{(Left column.) The derivative of the deflection $u$
  with respect to $x$. The presence of a singularity
  at $y=1/4$---due to a jump in the applied normal moment---is
  evident in the two topmost figures but not so much in the two
  bottom figures. (Middle column.) The meshes corresponding to the fifth adaptive refinement
  in the Nitsche method for different values of $\varepsilon^r$. 
  If $\varepsilon^r$ is small enough, the
  estimators successfully discard the lower singularity at $y=1/4$
  and focus instead
  on the singularity at $y=3/4$ caused by a jump in the
  Kirchhoff shear force.
  (Right column.) The meshes corresponding to the fifth adaptive refinement in the classical method
  for different values of $\varepsilon^r$.
  The estimators of the classical method 
  remain dominant near the lower singularity
  for small values of $\varepsilon^r$ due
  to the estimators scaling as $\mathcal{O}(1/\varepsilon^r)$.
  The source code for reproducing the example is available in~\cite{zenodo3}.
  }
  \label{fig:adaptcomp}
\end{figure}

\bibliographystyle{siamplain}
\bibliography{bibliography,plates}

\begin{thebibliography}{10}

\bibitem{arnold19_hellan_herrm_johns}
{\sc D.~Arnold and S.~Walker}, {\em {The Hellan-Herrmann-Johnson method with
  curved elements}},  (2019), \url{https://arxiv.org/abs/1909.09687}.

\bibitem{BdVNS-I}
{\sc L.~Beir\~ao~da Veiga, J.~Niiranen, and R.~Stenberg}, {\em A family of
  {$C^0$} finite elements for {K}irchhoff plates. {I}. {E}rror analysis}, SIAM
  J. Numer. Anal., 45 (2007), pp.~2047--2071,
  \url{https://doi.org/10.1137/06067554X}.

\bibitem{BdVNS-II}
{\sc L.~Beir\~ao~da Veiga, J.~Niiranen, and R.~Stenberg}, {\em A family of
  {$C^0$} finite elements for {K}irchhoff plates. {II}. {N}umerical results},
  Comput. Methods Appl. Mech. Engrg., 197 (2008), pp.~1850--1864,
  \url{https://doi.org/10.1016/j.cma.2007.11.015}.

\bibitem{blaauwendraad10_plates_fem}
{\sc J.~Blaauwendraad}, {\em Plates and Fem}, Springer Netherlands, 2010,
  \url{https://doi.org/10.1007/978-90-481-3596-7}.

\bibitem{brenner-2012-isopar-c}
{\sc S.~Brenner, M.~Neilan, and L.-Y. Sung}, {\em {Isoparametric $C^0$ Interior
  Penalty Methods for Plate Bending Problems on Smooth Domains}}, Calcolo, 50
  (2012), pp.~35--67, \url{https://doi.org/10.1007/s10092-012-0057-1}.

\bibitem{ciarlet}
{\sc P.~G. Ciarlet}, {\em The finite element method for elliptic problems},
  vol.~4 of Studies in Mathematics and its Applications, North-Holland
  Publishing Co., 1978.

\bibitem{ciarlet-2002-finit-elemen}
{\sc P.~G. Ciarlet}, {\em {The Finite Element Method for Elliptic Problems}},
  Society for Industrial and Applied Mathematics, 2002,
  \url{https://doi.org/10.1137/1.9780898719208}.

\bibitem{dominguez-2008-algor}
{\sc V.~Dom\'{\i}nguez and F.-J. Sayas}, {\em {Algorithm 884: A Simple Matlab
  Implementation of the Argyris Element}}, ACM Trans. Math. Softw., 35 (2008),
  \url{https://doi.org/10.1145/1377612.1377620}.

\bibitem{FS}
{\sc K.~Feng and Z.-C. Shi}, {\em Mathematical theory of elastic structures},
  Springer-Verlag, Berlin; Science Press, Beijing, 1996.

\bibitem{Friedrichs}
{\sc K.~Friedrichs}, {\em Die {R}andwert-und {E}igenwertprobleme aus der
  {T}heorie der elastischen {P}latten. ({A}nwendung der direkten {M}ethoden der
  {V}ariationsrechnung)}, Math. Ann., 98 (1928), pp.~205--247,
  \url{https://doi.org/10.1007/BF01451590}.

\bibitem{tom_gustafsson_2020_3925365}
{\sc T.~Gustafsson}, {\em kinnala/kirchhoff-nitsche-ex1 v1}, 2020,
  \url{https://doi.org/10.5281/zenodo.3925365}.

\bibitem{zenodo2}
{\sc T.~Gustafsson}, {\em kinnala/kirchhoff-nitsche-ex2 v1}, 2020,
  \url{https://doi.org/10.5281/zenodo.3925367}.

\bibitem{zenodo3}
{\sc T.~Gustafsson}, {\em kinnala/kirchhoff-nitsche-ex3 v1}, 2020,
  \url{https://doi.org/10.5281/zenodo.3925375}.

\bibitem{gustafsson-2020-kinnal-fem}
{\sc T.~Gustafsson and G.~D. McBain}, {\em kinnala/scikit-fem 1.0.0}, 2020,
  \url{https://doi.org/10.5281/zenodo.3773438}.

\bibitem{gustafsson-2018-poster-estim}
{\sc T.~Gustafsson, R.~Stenberg, and J.~Videman}, {\em {A Posteriori Estimates
  for Conforming Kirchhoff Plate Elements}}, SIAM J. Sci. Comput., 40 (2018),
  pp.~A1386--A1407, \url{https://doi.org/10.1137/17m1137334}.

\bibitem{juntunen2009nitsche}
{\sc M.~Juntunen and R.~Stenberg}, {\em Nitsche’s method for general boundary
  conditions}, Math. Comput., 78 (2009), pp.~1353--1374.

\bibitem{Kirchhoff1850}
{\sc G.~Kirchhoff}, {\em {{\"U}ber das Gleichgewicht und die Bewegung einer
  elastischen Scheibe}}, {J. Reine. Angew. Math.}, 40 (1850), pp.~51--88.

\bibitem{love1888xvi}
{\sc A.~Love}, {\em {XVI. The small free vibrations and deformation of a thin
  elastic shell}}, Philos. T. Roy. Soc. A,  (1888), pp.~491--546.

\bibitem{LJS}
{\sc N.~L\"{u}then, M.~Juntunen, and R.~Stenberg}, {\em An improved a priori
  error analysis of {N}itsche's method for {R}obin boundary conditions}, Numer.
  Math., 138 (2018), pp.~1011--1026,
  \url{https://doi.org/10.1007/s00211-017-0927-1}.

\bibitem{NH}
{\sc J.~Ne{\v{c}}as and I.~Hlav{{\'a}}{\v{c}}ek}, {\em {Mathematical Theory of
  Elastic and Elasto-Plastic Bodies: An Introduction}}, vol.~3 of Studies in
  Applied Mechanics, Elsevier Scientific Publishing Co., Amsterdam-New York,
  1980.

\bibitem{nitsche1971variationsprinzip}
{\sc J.~A. Nitsche}, {\em {{\"U}ber ein Variationsprinzip zur L{\"o}sung von
  Dirichlet-Problemen bei Verwendung von Teilr{\"a}umen, die keinen
  Randbedingungen unterworfen sind}}, in Abhandlungen aus dem mathematischen
  Seminar der Universit{\"a}t Hamburg, vol.~36, Springer, 1971, pp.~9--15.

\bibitem{renard-2020-getfem}
{\sc Y.~Renard and K.~Poulios}, {\em {GetFEM: Automated FE modeling of
  multiphysics problems based on a generic weak form language}},  (2020),
  \url{https://hal.archives-ouvertes.fr/hal-02532422} .

\bibitem{valdmanc1}
{\sc J.~Valdman}, {\em {MATLAB Implementation of C1 Finite Elements:
  Bogner-Fox-Schmit Rectangle}}, in Parallel Processing and Applied
  Mathematics, R.~Wyrzykowski, E.~Deelman, J.~Dongarra, and K.~Karczewski,
  eds., Cham, 2020, Springer International Publishing, pp.~256--266.

\bibitem{zhang19_accur}
{\sc J.~Zhang, C.~Zhou, S.~Ullah, Y.~Zhong, and R.~Li}, {\em {Accurate Bending
  Analysis of Rectangular Thin Plates With Corner Supports By a Unified Finite
  Integral Transform Method}}, Acta Mech., 230 (2019), pp.~3807--3821.

\bibitem{ZT}
{\sc O.~C. Zienkiewicz and R.~L. Taylor}, {\em The Finite Element Method.
  {V}olume 2: Solid Mechanics}, Butterworth-Heinemann, Oxford, fifth~ed., 2000.

\end{thebibliography}

\end{document}